\documentclass[11pt]{article}%
\usepackage{amssymb}
\usepackage{amsmath}
\usepackage{mitpress}
\usepackage{graphicx}
\usepackage{amsfonts}%
\providecommand{\U}[1]{\protect\rule{.1in}{.1in}}
\newtheorem{theorem}{Theorem}[subsection]

\newtheorem{lemma}[theorem]{Lemma}

\newtheorem{proposition}[theorem]{Proposition}
\newtheorem{remark}[theorem]{Remark}

\newenvironment{proof}[1][Proof]{\textbf{#1.} }{\ \rule{0.5em}{0.5em}}
\newdimen\dummy
\dummy=\oddsidemargin
\begin{document}

\title{\textsc{The translation operator for self-projective coalgebras}}
\author{William Chin\\DePaul University, Chicago Illinois 60614}
\maketitle

\begin{abstract}
We describe the transpose operator for self-projective and symmetric
coalgebras in terms of the syzygy and Nakayama functors.

\end{abstract}

\section{Introduction and Preliminaries}

The Auslander-Reiten (AR) translation plays a fundamental role in the
representation theory of coalgebras and algebras. Given a right comodule $M$,
under certain finiteness conditions, there exists an almost split sequence%
\[
0\rightarrow M\rightarrow E\rightarrow\tau(M)\rightarrow0
\]
where $\tau(M)$ is called the translation of $M$. For many right semiperfect
coalgebras, the almost split sequence exists for all non-projective
finite-dimensional right comodules $M$ [CKQ][CKQ2]. In these case the AR
quiver exists, and the translation induces an important special type of
automorphism of the stable quiver. An introduction to AR theory can be found
the books [ARS][ASS].

A coalgebra is said to be right self-projective if it is projective as a right
comodule over itself. A coalgebra is simply called self-projective if it is
both left and right self-projective. Self-projective coalgebras are also known
as quasi-co-Frobenius (qcF) coalgebras. A coalgebra $C$ is said to be\textit{
right co-Frobenius} if it embeds to\ its $\mathbb{\Bbbk}$-linear dual $DC$ as
a right $DC$-module. Facts and references concerning these types of coalgebras
may be found in [DNR]. Basic information and references concerning the
representation theory of coalgebras and path coalgebras of quivers can be
found in [Ch].

A coalgebra $C$ is said to be \textit{symmetric} [CDN] if it is isomorphic to
its $\mathbb{\Bbbk-}$linear dual $\mathrm{D}C$ as a a $\mathrm{D}%
C$,$\mathrm{D}C$-bimodule. A symmetric coalgebra is necessarily semiperfect,
i.e. it indecomposable injectives are finite-dimensional. It turns out that if
$C$ is a symmetric coalgebra, then its basic coalgebra is also symmetric. In
this article we focus on self-projective and symmetric coalgebras. In these
cases, the translation operator takes a particularly nice form. We show that
the translation given by $\Omega^{-2}\nu=\nu\Omega^{-2}$ for self-projective
coalgebras where $\nu$ is the Nakayama functor and $\Omega^{-2}$ is the second
cosyzygy. In the symmetric case $\nu$ is naturally isomorphic to the identity
functor on finite-dimensional comodules.

We present an example of a pointed self-projective serial coalgebra which is
not symmetric and compute the translation, Nakayama functor, and syzygies. The
coalgebra is a subcoalgebra of a quiver of type $A_{\infty}^{\infty}.$ We then
summarize some results from [Ch2] on representations of coordinate Hopf
algebra $C=\mathbb{\Bbbk}_{\zeta}[SL(2)]$ at a root of unity where
$\mathbb{\Bbbk}$ is a field of characteristic zero. Here we show that $C$ is a
symmetric coalgebra. The results concerning the translation are illustrated in
the almost split sequences and AR quiver.

A theory in a very different direction, but which also enables the computation
of the translation and almost split sequences on the level of Grothendieck
groups of comodules was recently developed in [CS]. There the theory of
Coxeter transformations is developed, with a focus on hereditary coalgebras.
In the case of path coalgebras with locally finite quivers, it is shown that
the image of the translation of a finite-dimensional indecomposable comodule
in the Grothendieck group is given by the Coxeter transformation, which is
determined by a possibly infinite invertible Cartan matrix.

\bigskip

We shall let $C$ denote a coalgebra over the fixed base field $\mathbb{\Bbbk}%
$. A right $C$-comodule $M$ is said to be \textit{quasifinite} if
$\mathrm{Hom}^{C}(F,M)$ is finite-dimensional for all finite-dimensional
comodules $F$. $M$ is said to be \textit{quasifinitely copresented} if it has
an injective copresentation $0\rightarrow N\rightarrow I_{0}\rightarrow I_{1}$
with quasifinite injectives $I_{i}.$ $\mathcal{M}^{C}$, $\mathcal{M}_{f}%
^{^{C}},$ $\mathcal{M}_{q}^{^{C}}$ $\mathcal{M}_{qc}^{^{C}}$ shall denote the
category of right $C$-comodules, finite-dimensional right $C$-comodules,
quasifinite right $C$-comodules and quasifinitely copresented right
$C$-comodules, respectively. The $\Bbbk-$linear dual $\operatorname*{Hom}%
_{\Bbbk}(-,\Bbbk)$\textrm{:}$\mathcal{M}_{f}^{^{C}}\rightarrow\mathcal{M}%
_{f}^{^{C^{op}}}$ is denoted by \textrm{D:}$\mathcal{M}_{f}^{^{C}}%
\rightarrow\mathcal{M}_{f}^{^{C^{op}}}.$ The cohom functor $h_{C}%
(-,-):\mathcal{M}_{q}^{C}\rightarrow\mathcal{M}_{\Bbbk}$ is defined to be
\[
h_{C}(M,N)=\lim_{\rightarrow}\mathrm{DHom}(N_{\lambda},M)
\]
for $M,N\in\mathcal{M}^{C}$ where $N_{\lambda}$ is the direct system of
finite-dimensional comodules of $N$ indexed by $\lambda.$

\subsection{Symmetric and Self-projective coalgebras}

Let $N$ be a quasi-finite right $C$-comodule and let $I(N)$ denote its
injective envelope. Recall [CKQ] that the functor $^{\ast}$: $\mathcal{M}%
_{q}^{C}\rightarrow$ $^{C}\mathcal{M}$ is defined by $h_{C}(-,C).$ The functor
$^{\ast}$ restricts to a duality on quasifinite injectives and is the
coalgebraic version of the functor $\mathrm{Hom}_{R}(-,R)$ for $R$-modules.
The transpose operator Tr is defined in [CKQ] via the minimal injective
copresentation
\[
0\rightarrow\mathrm{Tr}N\rightarrow I_{1}^{\ast}\rightarrow I_{0}^{\ast}%
\]
where $0\rightarrow N\rightarrow I_{0}\rightarrow I_{1}$ is a minimal
injective copresentation of $N.$ In general, the translation operator is
defined to be DTr$(N)$ on comodules $N$ such that Tr$N$ is finite-dimensional.
Here we are concerned with semiperfect coalgebras and will assume comodules
are finite-dimensional so that the translation is always defined. The cosyzygy
operator $\Omega^{-1}$ is defined as the cokernel in the exact sequence
\[
0\rightarrow N\rightarrow I(N)\rightarrow\Omega^{-1}(N)\rightarrow0.
\]
If $C$ is right semiperfect, the syzygy operator $\Omega$ is defined by the
exact sequence
\[
0\rightarrow\Omega(N)\rightarrow P(N)\rightarrow N\rightarrow0
\]
where $P(N)$\ denotes the projective cover of $N$. For semiperfect coalgebras
we note that these functors are inverse equivalences between the categories of
finite-dimensional comodules modulo injectives and projectives respectively:
\[
\overline{\mathcal{M}}_{f}^{C}\overset{\Omega}{\underset{\Omega^{-1}%
}{\leftrightarrows}}\underline{\mathcal{M}}_{f}^{C}.
\]

We collect some facts about comodule functors. The reader is referred to [CDN]
for a development of the theory of symmetric coalgebras, and the book [DNR]
for semiperfect and self-projective (i.e. quasi-co-Frobenius) coalgebras. Here
we just point out that symmetric$\Rightarrow$self-projective$\Rightarrow
$semiperfect (one-sided and two-sided versions). In addition, it is recently
be shown that one-sided self-projective implies semiperfect [Iov].

\begin{lemma}
Let $C=\oplus_{i}P_{i}$ be a right self-projective coalgebra with
projective-injective indecomposable right coideals $P_{i}$. Then for a
finite-dimensional comodule $N,$ \newline(a) $N^{\ast}=\oplus_{i}%
\mathrm{D\operatorname*{Hom}}(P_{i},N)$\newline(b) $^{\ast}$: $\mathcal{M}%
_{q}^{C}\rightarrow$ $^{C}\mathcal{M}$ is an exact functor. \newline(c) $\nu:$
$\mathcal{M}_{f}^{C}\rightarrow$ $^{C}\mathcal{M}_{f}$ is an exact functor.
\end{lemma}

\begin{proof}
The first statement follows directly from the definition of the cohom functor
since $C$ is plainly the direct limit of the finite-dimensional projective
summands $P_{i}$. The remaining assertions follows directly, since the $P_{i}$
are projective and the $\mathbb{\Bbbk-}$linear dual \textrm{D} is exact.
\end{proof}

\begin{proposition}
(a) For a left self-projective coalgebra $C$, the translation is naturally
isomorphic to $\Omega^{-2}\nu$. \newline(b) For a right self-projective
coalgebra $C$, the translation is naturally isomorphic to $\nu\Omega^{-2}$.
\newline(c) For a symmetric coalgebra $C$, the translate is naturally
isomorphic to $\Omega^{-2}$.
\end{proposition}

\begin{proof}
Let $N$ be an indecomposable non-injective finite-dimensional right
$C$-comodule. Then the transpose is defined by $0\rightarrow$Tr$N\rightarrow
I_{1}^{\ast}\rightarrow I_{0}^{\ast}$ where $0\rightarrow N\rightarrow
I_{0}\rightarrow I_{1}$ is a minimal injective copresentation of $N.$ We
obtain using exactness of D,
\[
0\rightarrow\mathrm{D}N^{\ast}\rightarrow\mathrm{D}I_{0}^{\ast}\rightarrow
\mathrm{D}I_{1}^{\ast}\rightarrow\mathrm{DTr}N\rightarrow0,
\]
i.e. $0\rightarrow\nu N\rightarrow\nu I_{0}\rightarrow\nu I_{1}\rightarrow
\mathrm{DTr}N\rightarrow0$. Since the $I_{i}^{\ast}$ are injective and hence
projective left comodules, the $\nu I_{i}^{\ast}$ are injective, and we have a
minimal injective resolution of $\nu N.$ Thus $\Omega^{-2}\nu(N)=\mathrm{DTr}%
N.$ This proves (a).

To prove (b) we have by definition of the cosyzygy,%
\[
0\rightarrow N\rightarrow I_{0}\rightarrow I_{1}\rightarrow\Omega
^{-2}(N)\rightarrow0
\]
and by exactness of $^{\ast}$,
\[
0\rightarrow(\Omega^{-2}N)^{\ast}\rightarrow I_{1}^{\ast}\rightarrow
I_{0}^{\ast}\rightarrow N^{\ast}\rightarrow0
\]
whence $(\Omega^{-2}N)^{\ast}\cong\mathrm{Tr}N;$ applying D yields the result.

Assume next that $C$ is symmetric. Let $R$ denote the dual algebra D$C.$ By
definition, $C\ $embeds in $R$ as a $R,R$-bimodule and by [CDN, Theorem 5.3],
we have
\[
\operatorname*{Hom}\nolimits_{\mathbb{\Bbbk}}(N,\mathbb{\Bbbk})\cong%
\operatorname*{Hom}\nolimits_{R}(N,R),
\]
an isomorphism of right $R$-modules, which is natural in $N$. This can be seen
by noting that $C$ embeds in $R$ as a left $R$-module as Rat$(_{R}R),$ the
largest right rational left $R$-submodule of $R$. Then $\operatorname*{Hom}%
\nolimits_{\mathbb{\Bbbk}}(N,\mathbb{\Bbbk})\cong\operatorname*{Hom}%
\nolimits_{R}(N,C)\cong\operatorname*{Hom}\nolimits_{R}(N,Rat(R))\cong%
\operatorname*{Hom}\nolimits_{R}(N,R)$ (the first isomorphism sends $f$ to
$(id_{N}\otimes f)\rho_{N}$).

We need to show that the composition \textrm{D}$(-)^{\ast}$ is naturally
isomorphic to the identity functor. By the Lemma, $N^{\ast}=\oplus
\mathrm{D}\operatorname*{Hom}(P_{i},N)$ where the $P_{i}$ are projective
indecomposable right subcomodules of $C$. Since $N$ is of finite length,
$\operatorname*{Hom}(P_{i},N)$ is nonzero for only finitely many $i$. It
follows (by taking coefficient space of $\oplus P_{i}$) that $\mathrm{D}%
\operatorname*{Hom}_{\mathrm{D}C}(C,N)\cong\mathrm{D}\operatorname*{Hom}%
_{\mathrm{D}C}(F,N)$ for some finite-dimensional subcoalgebra $F\subset C$.
Thus may assume $C$ is finite dimensional and we have isomorphisms
\begin{align*}
N^{\ast}  &  =\mathrm{D}\operatorname*{Hom}\nolimits_{R}(C,N)\\
&  \cong\mathrm{D}\operatorname*{Hom}\nolimits_{R}(\mathrm{D}N,R)\\
&  \cong\mathrm{D}\operatorname*{Hom}\nolimits_{\Bbbk}(\mathrm{D}N,\Bbbk)\\
&  =\mathrm{D}^{3}(N)\\
&  \cong\mathrm{D}N,
\end{align*}
which are natural in $N.$ Therefore $\nu$ is isomorphic to the identity
functor on $\mathcal{M}_{f}^{C}.$ This completes the proof of (c).
\end{proof}

\bigskip

Thus, by results of [CKQ], if $N$ is an indecomposable non-injective
finite-dimensional comodule over a symmetric coalgebra, then the almost split
sequence starting at $N$ ends at $\Omega^{-2}(N)$. Dually, if $M$ is an
indecomposable non-injective finite-dimensional comodule over a symmetric
coalgebra, then the almost split sequence ending at $M$ starts at $\Omega
^{2}(N).$

\section{Examples}

\subsection{$\mathbb{A}_{\infty}^{\infty}$}

Let $Q$ be the quiver of type $\mathbb{A}_{\infty}^{\infty}$ with vertices
indexed by $\mathbb{Z}$ and arrows $a_{i}:i\rightarrow i+1$ for all
$i\in\mathbb{Z}$. Fix $n>0$ and consider the subcoalgebra $C$ of the path
coalgebra $\mathbb{\Bbbk}Q$ spanned by paths of length at most $n$. Thus $C$
is the degree $n$ term of the coradical filtration of $\mathbb{\Bbbk}Q$.

The coalgebra $C$ is a serial coalgebra [CG2] and each of its
finite-dimensional indecomposable representations is isomorphic to the right
coideal generated by a path of some fixed length $\ell\leq n.$ In detail, fix
integers $i\leq j$ and let $V_{i,j}$ be the right coideal spanned by subpaths
$a_{j}\cdots a_{t}$, $i\leq t\leq j,$ all of which end at $j$ (paths are
written from right to left). Then $V_{i,j}$ is the right coideal generated by
the path $a_{j}\cdots a_{i}$ of length $j-i\leq n$ where $V_{i,i}$ is regarded
as the simple comodule at the vertex $i.$ As a representation of $Q,$
$V_{i,j}$ has a one-dimensional vector space at the vertices between $i$ and
$j$ with identity maps assigned to the arrows. In particular $S(i)=V_{i,i}$ is
the simple right comodule at the vertex $i.$ Dually, let $U_{i,j}$ be the left
coideal spanned by the subpaths $a_{t}\cdots a_{i}$ of $a_{j}\cdots a_{i}.$
The $U_{i,j}$ are the indecomposable left comodules, which correspond to
indecomposable representations of $Q^{op},$ and $U_{i,i}$ is the simple left
comodule at the vertex $i.$ Clearly, $U_{i,j}=\mathrm{D}V_{i,j}.$ It is easy
to see that $I_{i}=V_{i-n,i}$ is both the injective envelope of $S(i)$ and the
projective cover of $S(i-n).$ Evidently, $C$ is a self-projective coalgebra.
But $C$ is not a symmetric coalgebra because the Nakayama functor is not
isomorphic to the identity as we show presently.

\begin{proposition}
(a) $\Omega^{-2}(V_{i,j})=V_{i-n-1,j-n-1}$\newline(b) $\nu(V_{i,j}%
)=V_{i+n,j+n}$\newline(c) \textrm{DTr}$(V_{i,j})=V_{i-1,j-1}$\newline(d)
$\Omega^{-2}(U_{i,j})=V_{i+n+1,j+n+1}$\newline(e) $\nu(U_{i,j})=V_{i-n,j-n}%
$\newline(f) \textrm{DTr}$(U_{i,j})=V_{i+1,j+1}$
\end{proposition}

\begin{proof}
We have the injective copresentation
\[
0\rightarrow V_{i,j}\rightarrow V_{j-n,j}\rightarrow V_{i-1-n,i-1}%
\rightarrow\Omega^{-2}(V_{i,j})\rightarrow0
\]

which yields $\Omega^{-2}(V_{i,j})=V_{i-n-1,j-n-1}$ by inspection. This proves
(a). Dualizing, we have
\[
0\rightarrow\operatorname*{Tr}V_{i,j}\rightarrow V_{i-1-n,i-1}^{\ast
}\rightarrow V_{j-n,j}^{\ast}\rightarrow V_{i,j}^{\ast}\rightarrow0
\]
which we rewrite as%

\[
0\rightarrow\operatorname*{Tr}V_{i,j}\rightarrow U_{i-1,i-1+n}\rightarrow
U_{j,j+n}\rightarrow V_{i,j}^{\ast}\rightarrow0
\]
By inspection this yields $\operatorname*{Tr}V_{i,j}=U_{i-1,j-1}$ and
$V_{i,j}^{\ast}=U_{i+n,j+n}.$ Assertions (b) and (c) follow immediately. The
dual assertions (d-f) are similar.
\end{proof}

The almost split sequences are just as for Nakayama algebras [ARS] as follows:%
\[
0\rightarrow V_{i,j}\rightarrow V_{i-1,j}\oplus V_{i,j-1}\rightarrow
V_{i-1,i-1}\rightarrow0
\]
if $n>i-j>0.$ The boundary sequences are%
\[
0\rightarrow S(i)\rightarrow V_{i-1,i}\rightarrow S(i-1)\rightarrow0
\]
and
\[
0\rightarrow V_{i-n+1,i-1}\rightarrow I_{i}\oplus V_{i-n+1,i-1}\rightarrow
V_{i-n,i-1}\rightarrow0.
\]
The AR quiver for $n=4$ is shown below. The stable AR quiver is of type
$\mathbb{ZA}_{n}.$ The translation is denoted by $\dashrightarrow.$
\[
\cdot\cdot\cdot%
\begin{array}
[c]{ccccccccccccccc}
& S(1) &  & \dashrightarrow &  & S(0) &  & \dashrightarrow &  & S(-1) &  &
\dashrightarrow &  & S(-2) & \\
\nearrow &  & \searrow &  & \nearrow &  & \searrow &  & \nearrow &  & \searrow
&  & \nearrow &  & \searrow\\
& \dashrightarrow &  & V_{0,1} &  & \dashrightarrow &  & V_{-1,0} &  &
\dashrightarrow &  & V_{-2,-1} &  & \dashrightarrow & \\
\searrow &  & \nearrow &  & \searrow &  & \nearrow &  & \searrow &  & \nearrow
&  & \searrow &  & \nearrow\\
& V_{0,2} &  & \dashrightarrow &  & V_{-1,1} &  & \dashrightarrow &  &
V_{-2,0} &  & \dashrightarrow &  & V_{-3,-1} & \\
\nearrow &  & \searrow &  & \nearrow &  & \searrow &  & \nearrow &  & \searrow
&  & \nearrow &  & \searrow\\
& \dashrightarrow &  & V_{-1,2} &  & \dashrightarrow &  & V_{-2,1} &  &
\dashrightarrow &  & V_{-3,0} &  & \dashrightarrow & \\
\searrow &  & \nearrow &  & \searrow &  & \nearrow &  & \searrow &  & \nearrow
&  & \searrow &  & \nearrow\\
& V_{-1,3} & \rightarrow & I_{3} & \rightarrow & V_{-2,2} & \rightarrow &
I_{2} & \rightarrow & V_{-3,1} & \rightarrow & I_{1} & \rightarrow & V_{-4,0}
&
\end{array}
\cdot\cdot\cdot
\]

\subsection{Quantum SL(2) at a root of 1}

Assume the base field $\Bbbk$ is of characteristic zero. Henceforth, let
$C=\Bbbk_{\zeta}[SL(2)]$ be the $q$-analog of the coordinate Hopf algebra of
$SL(2)$, where $q$ is specialized to a root of unity $\zeta$ of odd order
$\ell$.

The coalgebra $C$ has the following presentation. The algebra generators are
$a,b,c,d,$ with relations
\begin{align*}
ba  &  =\zeta ab\\
db  &  =\zeta db\\
ca  &  =\zeta ac\\
bc  &  =cb\\
ad-da  &  =(\zeta-\zeta^{-1})bc\\
ad-\zeta^{-1}bc  &  =1
\end{align*}
and with Hopf algebra structure further specified by%
\begin{align*}
\Delta%
\begin{pmatrix}
a & b\\
c & d
\end{pmatrix}
&  =%
\begin{pmatrix}
a & b\\
c & d
\end{pmatrix}
\otimes%
\begin{pmatrix}
a & b\\
c & d
\end{pmatrix}
\\
\varepsilon%
\begin{pmatrix}
a & b\\
c & d
\end{pmatrix}
&  =%
\begin{pmatrix}
1 & 0\\
0 & 1
\end{pmatrix}
,\text{ }S%
\begin{pmatrix}
a & b\\
c & d
\end{pmatrix}
=%
\begin{pmatrix}
d & -\zeta b\\
-\zeta^{-1}c & a
\end{pmatrix}
\end{align*}

By general coalgebra theory (see [Ch]), $C$ is the direct sum of blocks, each
determined by an equivalence class (under linkage) of simple comodules. For
quantum $SL(2)$, each simple is labelled by a nonnegative integer, called a
\textit{highest weight}. In [Ch2] we see that there are precisely $\ell-1$
nontrivial blocks $C_{0},...,C_{\ell-2}$ where $C_{r}$ contains the simple
comodule of highest weight $r,$ $0\leq r<\ell-2.$ The other simples in $C_{r}$
are obtained via $\ell$-reflections. There are an infinite set of trivial
blocks, whose simple modules are projective and injective and whose quivers
therefore consist of a single vertex. Thus in some sense the representation
theory of $C$ is reduced to the nontrivial blocks, which turn out to be
Morita-Takeuchi equivalent to each other. A nontrivial block of the basic
coalgebra is described as a subcoalgebra of the path coalgebra of its quiver
in the next result.

\begin{theorem}
[{[Ch2]}]The basic coalgebra $B$ of $C_{r}$ is the subcoalgebra of path
coalgebra of the quiver
\[
0\overset{b_{0}}{\underset{a_{0}}{\leftrightarrows}}1\overset{b_{1}}%
{\underset{a_{1}}{\leftrightarrows}}2\overset{b_{2}}{\underset{a_{2}%
}{\leftrightarrows}}\cdot\cdot\cdot
\]
spanned by the by group-likes $g_{i}$ corresponding to vertices and arrows
$a_{i},b_{i}$ together with coradical degree two elements
\begin{align*}
d_{0}  &  :=b_{0}a_{0}\\
d_{i+1}  &  :=a_{i}b_{i}+b_{i+1}a_{i+1},\text{ }i\geq0.
\end{align*}

\end{theorem}

\begin{proposition}
$C$ is a \textit{symmetric} coalgebra.
\end{proposition}

\begin{proof}
By a result of Radford, there is a distinguished group-like element in $C$
(see [DNR, p. 197]), which is trivial if and only if $C$ is unimodular. There
is a unique one-dimensional comodule, namely the trivial comodule, so $C$ is
unimodular. This implies that the distinguished group-like element in $C$ is
just the identity element. The square of the antipode of $U_{\zeta}$ is inner
by the element $K^{2}\in U_{\zeta}\subset DC,$ and it follows by duality that
the square of the antipode of $C$ is given by $S^{2}(c)=K^{-2}\rightharpoonup
c\leftharpoonup K^{2}$ (a fact that can be checked directly). Thus $C$ is
unimodular and is inner by an element of $DC,$ and the conclusion now follows
from [CDN].
\end{proof}

\begin{remark}
The conclusion can be similarly shown to hold more generally for quantized
coordinate algebras, using the fact that $S^{2}$ is inner in the generic
quantized enveloping algebra $U_{\zeta}(\mathfrak{g})$ where $\mathfrak{g}$ is
complex semisimple Lie algebra, see e.g., [Ja] 4.4.
\end{remark}

The fact that each nontrivial block $B$ is a symmetric coalgebra follows from

\begin{lemma}
The basic coalgebra of a symmetric coalgebra $C$ is symmetric.
\end{lemma}

\begin{proof}
By definition [CDN], there is a $\mathrm{D}C$-bimodule embedding
$C\rightarrow\mathrm{D}C.$ By [CG] the basic coalgebra $B$ of $C$ is of the
form $eCe$ where $e\in\mathrm{D}C$ (left and right "hit" actions). One can
easily check that $\mathrm{D}B\cong e(\mathrm{D}C)e,$ so that the coalgebra
$eCe$ embeds in $\mathrm{D}(eCe)$ as a $\mathrm{D}B,\mathrm{D}B$-bimodule by restriction.
\end{proof}

\begin{remark}
A symmetrizing form (see [CDN]) $\phi:B\rightarrow\Bbbk$ can be defined by
$\phi(d_{i})=1$ and $\phi(a_{i})=\phi(b_{i})=\phi(g_{i})=0$ for all $i.$ This
gives another proof that each nontrivial block $B$ is a symmetric coalgebra.
This form is an associative, symmetric, nondegenerate, D$B$-balanced form as required.
\end{remark}

The indecomposable injective right $B$-comodules are the coideals
\begin{align*}
I_{n}  &  =\Bbbk g_{n}+\Bbbk a_{n-1}+\Bbbk b_{n}+\Bbbk d_{n}\\
I_{0}  &  =\Bbbk g_{0}+\Bbbk b_{0}+\Bbbk d_{0}%
\end{align*}
$(n\geq1)$ and they are also projective comodules. Clearly rad$I_{n}=\Bbbk
g_{n}+\Bbbk a_{n-1}+\Bbbk b_{n}$ and rad$I_{0}=\Bbbk g_{0}+\Bbbk b_{0}$. Let
$S(n)$ the simple right comodule corresponding to the vertex $g_{n}$, for all
$n\in\mathbb{N}.$ The non-injective finite-dimensional comodules are described
in [Ch2] as \textit{string comodules }and can be computed using cosyzygies and syzygies:

\begin{proposition}
[{[Ch2]}]Every finite-dimensional non-injective indecomposable $B$-comodule is
in the $\Omega^{\pm}$-orbit of some simple comodule.
\end{proposition}

The almost split sequences and AR quiver for the category of
finite-dimensional $B$-comodules are described starting with the sequences
having an injective-projective comodule $I_{n}$ in the middle term. These are
precisely the sequences (see e.g. [ARS, p. 169])
\begin{equation}
0\rightarrow\mathrm{rad}(I_{n})\rightarrow\frac{\mathrm{rad}(I_{n}%
)}{\mathrm{soc}(I_{n})}\oplus I_{n}\rightarrow\frac{I_{n}}{\mathrm{soc}I_{n}%
}\rightarrow0 \tag*{(1)}\label{1}%
\end{equation}
with $n\in\mathbb{N}$. These sequences can be rewritten as
\begin{equation}
0\rightarrow\Omega(S(n))\rightarrow S(n-1)\oplus S(n+1)\oplus I_{n}%
\rightarrow\Omega^{-1}(S(n))\rightarrow0 \tag*{(2)}\label{2}%
\end{equation}
where $S(-1)$ is declared to be $0$.

\begin{theorem}
[{[Ch2]}]Applying $\Omega^{i}$, $i\in\mathbb{Z},$ to the sequences \ref{2}
yields all almost split sequences for $\mathcal{M}_{f}^{B}.$
\end{theorem}

The sequences \ref{2} are precisely the almost split sequences with an
injective-projective in the middle term; the other sequences
\[
0\rightarrow\Omega^{i+1}(S(n))\rightarrow\Omega^{i}S(n-1)\oplus\Omega
^{i}S(n+1)\rightarrow\Omega^{i-1}(S(n))\rightarrow0
\]
we obtain when we apply $\Omega^{i},$ for a nonzero integer $i,\ $are the
almost split sequences without an injective-projective summand in the middle term.

We define the Auslander-Reiten (AR) quiver of a coalgebra to be the quiver
whose vertices are isomorphism classes of indecomposable comodules and whose
(here multiplicity-free) arrows are defined by the existence of an irreducible
map between indecomposables.

We finally display the AR quiver for $\mathcal{M}_{f}^{B}$.$\ $The stable AR
quiver (with all injectives deleted) consists of two components of type
$\mathbb{ZA}_{\infty}$ which are transposed by $\Omega$.%

\begin{gather*}
\cdot\cdot\cdot%
\begin{array}
[c]{ccccccccccccccc}
& \Omega^{3}S(0) &  & \dashrightarrow &  & \Omega S(0) & \rightarrow & I_{0} &
\rightarrow & \Omega^{-1}S(0) &  & \dashrightarrow &  & \Omega^{-3}S(0) & \\
\nearrow &  & \searrow &  & \nearrow &  & \searrow &  & \nearrow &  & \searrow
&  & \nearrow &  & \searrow\\
& \dashrightarrow &  & \Omega^{2}S(1) &  & \dashrightarrow &  & S(1) &  &
\dashrightarrow &  & \Omega^{-2}S(1) &  & \dashrightarrow & \\
\searrow &  & \nearrow &  & \searrow &  & \nearrow &  & \searrow &  & \nearrow
&  & \searrow &  & \nearrow\\
& \Omega^{3}S(2) &  & \dashrightarrow &  & \Omega S(2) & \rightarrow & I_{2} &
\rightarrow & \Omega^{-1}S(2) &  & \dashrightarrow &  & \Omega^{-3}S(2) & \\
\nearrow &  & \searrow &  & \nearrow &  & \searrow &  & \nearrow &  & \searrow
&  & \nearrow &  & \searrow\\
& \dashrightarrow &  & \Omega^{2}S(3) &  & \dashrightarrow &  & S(3) &  &
\dashrightarrow &  & \Omega^{-2}S(3) &  & \dashrightarrow & \\
\searrow &  & \nearrow &  & \searrow &  & \nearrow &  & \searrow &  & \nearrow
&  & \searrow &  & \nearrow\\
& \Omega^{3}S(4) &  & \dashrightarrow &  & \Omega S(4) & \rightarrow & I_{4} &
\rightarrow & \Omega^{-1}S(4) &  & \dashrightarrow &  & \Omega^{-3}S(4) &
\end{array}
\cdot\cdot\cdot\\
\cdot\\
\cdot\\
\cdot
\end{gather*}

\begin{gather*}
\cdot\cdot\cdot%
\begin{array}
[c]{ccccccccccccccc}
& \Omega^{2}S(0) &  & \dashrightarrow &  & S(0) &  & \dashrightarrow &  &
\Omega^{-2}S(0) &  & \dashrightarrow &  & \Omega^{-4}S(0) & \\
\nearrow &  & \searrow &  & \nearrow &  & \searrow &  & \nearrow &  & \searrow
&  & \nearrow &  & \searrow\\
& \dashrightarrow &  & \Omega S(1) & \rightarrow & I_{1} & \rightarrow &
\Omega^{-1}S(1) &  & \dashrightarrow &  & \Omega^{-3}S(1) &  & \dashrightarrow
& \\
\searrow &  & \nearrow &  & \searrow &  & \nearrow &  & \searrow &  & \nearrow
&  & \searrow &  & \nearrow\\
& \Omega^{2}S(2) &  &  &  & S(2) &  & \dashrightarrow &  & \Omega^{-2}S(2) &
& \dashrightarrow &  & \Omega^{-4}S(2) & \cdots\\
\nearrow &  & \searrow &  & \nearrow &  & \searrow &  & \nearrow &  & \searrow
&  & \nearrow &  & \searrow\\
& \dashrightarrow &  & \Omega S(3) & \rightarrow & I_{3} & \rightarrow &
\Omega^{-1}S(3) &  & \dashrightarrow &  & \Omega^{-3}S(3) &  & \dashrightarrow
& \\
\searrow &  & \nearrow &  & \searrow &  & \nearrow &  & \searrow &  & \nearrow
&  & \searrow &  & \nearrow\\
& \Omega^{2}S(4) &  & \dashrightarrow &  & S(4) &  & \dashrightarrow &  &
\Omega^{-2}S(4) &  & \dashrightarrow &  & \Omega^{-4}S(4) &
\end{array}
\cdot\cdot\cdot\\
\cdot\\
\cdot\\
\cdot
\end{gather*}

\end{document}